\documentclass{amsart}

\usepackage[english]{babel}
\usepackage{enumerate}
\usepackage[all]{xy}

\theoremstyle{plain}
\newtheorem{theorem}{Theorem}[section]
\newtheorem{lemma}[theorem]{Lemma}
\newtheorem{proposition}[theorem]{Proposition}

\theoremstyle{remark}
\newtheorem{problem}[theorem]{Problem}
\newtheorem{remark}[theorem]{Remark}

\newtheorem{example}{Example}

\newcommand{\Z}{\mathbb{Z}}
\newcommand{\Q}{\mathbb{Q}}

\DeclareMathOperator{\charac}{char}
\DeclareMathOperator{\Aut}{Aut}
\DeclareMathOperator{\GL}{GL}
\DeclareMathOperator{\SL}{SL}
\DeclareMathOperator{\Norm}{Norm}

\DeclareMathOperator{\End}{End}

\begin{document} 

\title[Free symmetric and unitary pairs]{Free symmetric and unitary pairs in 
division rings infinite-dimensional over their centers}
\author{Vitor O.~Ferreira}
\author{Jairo Z.~Gon\c{c}alves}
\address{Department of Mathematics, University of S\~{a}o Paulo, S\~{a}o Paulo, 
05508-090, Brazil}
\email{vofer@ime.usp.br, jz.goncalves@usp.br}
\thanks{Both authors were partially supported by FAPESP-Brazil, Proj.~Tem\'atico 
2009/52665-0. The second author was partially supported by Grant CNPq 
301.320/2011-0.}

\dedicatory{Dedicated to A.~I.~Lichtman, a pioneer in division rings.}

\begin{abstract}
Let $D$ be a division ring infinite-dimensional over its center $k$ with 
multiplicative group $D^{\times}$. We show 
that if $D$ belongs to certain families, there exist free symmetric and 
unitary pairs in $D^{\times}$ with respect to a $k$-involution on $D$.
\end{abstract}

\maketitle


\section{Introduction}\label{S:Intro}

Lichtman conjectured in \cite{Lichtman77} that if $D$ is a noncommutative 
division ring, then its multiplicative group will 
contain a noncyclic free subgroup. Division rings come, often, accompanied
by natural involutions (this is the case,
for instance, of division rings of fractions of group algebras which are
Ore domains). In \cite{GS06}, in their quest to prove an 
analogue of Lichtman's conjecture for a division ring with an involution, 
the authors showed that if the center $k$ of $D$ is a nonabsolute field of 
characteristic different from $2$ and $D$ is 
finite-dimensional over $k$, then the multiplicative group of $D$
contains a noncyclic free subgroup generated by symmetric or
unitary elements, except when $D$ is a quaternion algebra with an appropriate 
involution in each case. They were thus led to propose in \cite{GS11} that, 
except in some very special circumstances, a noncomutative division ring 
$D$ will contain, inside its multiplicative group, a free noncyclic 
subgroup which is generated by symmetric or unitary elements with respect to 
whichever involution that is present.

Not much is known when $D$ is infinite-dimensional over its center (see \cite{FGM05}), 
and this is the situation we are going to address in the present paper.

\medskip

In what follows, we shall adopt the following notation and definitions.

Given a field $k$ and a $k$-algebra $A$, a \emph{$k$-involution} on $A$ is
a $k$-linear operator $\ast$ of $A$ satisfying $(ab)^{\ast}=b^{\ast}a^{\ast}$ 
and $a^{\ast\ast}=a$, 
for all $a,b\in A$. An element $a \in A$ is said to be \emph{symmetric} if 
$a^{\ast}=a$, \emph{anti-symmetric} if $a^{\ast}=-a$, and \emph{unitary} 
if $aa^{\ast}=a^{\ast}a=1$. 

The group of invertible elements of $A$ will be denoted by $A^{\times}$. A pair
of elements of $A^{\times}$ will be said to be \emph{free} if the subgroup
they generate inside $A^{\times}$ is free of rank $2$.

By an \emph{involution} on a group $G$ one understands an anti-automorphism of $G$ of order $2$.
Thus, if $G$ is a group and $k$ is a field, then, given an involution
on $G$, its $k$-linear extension to the group algebra $kG$ is a $k$-involution which, in case $kG$ happens to be an Ore domain with division ring of quotients $D$, extends to a unique $k$-involution on $D$.
Such an involution on $D$ will be said to be \emph{induced} by the
involution on $G$.
 
In \cite{GS11}, the authors considered the case where $D$ is the division ring of quotients of 
the group algebra of a nonabelian torsion free nilpotent group $G$ over a field of 
characteristic different from $2$, with an involution that is induced by one on $G$. 
However, the study was not complete, since many involutions were 
not taken into account. Here we come back to the subject, exhausting almost all cases. 

Given two elements $\alpha$ and $\beta$ in a group, their commutator 
$\alpha^{-1}\beta^{-1}\alpha\beta$ will be denoted by $[\alpha,\beta]$.
By the \emph{Heisenberg group} we shall understand the free nilpotent
group of class $2$ generated by $2$ elements, that is, the group $\Gamma$ presented by
\begin{equation}\label{eq:heis}
\Gamma=\langle x,y : [[x,y],x]=[[x,y],y]=1\rangle.
\end{equation}
It is well known that
$\lambda=[x,y]$ generates the center of $\Gamma$ and has infinite order;
moreover, the quotient of 
$\Gamma$ by its center is a free abelian group of rank $2$. The elements of
$\Gamma$ can be written uniquely in the form $\lambda^ry^mx^n$, with $r,m,n\in\Z$.

The main result of the paper is the following.

\begin{theorem}\label{T:NilpD}
  Let $D$ be the division ring of quotients of the group algebra of the Heisenberg group $\Gamma$
  over a field $k$ of characteristic different from $2$ and let $\ast$ be a $k$-involution
  on $D$ which is induced by an involution on $\Gamma$. Then $D$ contains free symmetric
  and unitary pairs if the involution on $\Gamma$ is not of form $x^{\ast}=\zeta x$ and
  $y^{\ast}=\eta y$, for $\zeta,\eta$ in the center of $\Gamma$. In this last case, 
	free symmetric pairs exist if both $\zeta$ and $\eta$ are not odd powers of $[x,y]$,
	and free unitary pairs exist in the complementary case.
\end{theorem}

For the proof of Theorem~\ref{T:NilpD}, in Section~\ref{S:One}, the idea will be
to lift free pairs through specializations from $D$ to a suitable quaternion algebra where free pairs
of the desired form are known to exist. It  should be noted that in all cases
 explicit free pairs will be exhibited.

The inversion map $g\mapsto g^{-1}$ on a group $G$ will be called the
\emph{canonical involution} on $G$. In this special case, we can offer
the following improvement of Theorem~\ref{T:NilpD}.

\begin{theorem}\label{T:NilpDUsual}
  Let $G$ be a nonabelian torsion free nilpotent group, let $k$ be a field
  and let $D$ be the division ring of quotients of the group algebra $kG$. Then, if 
  $\charac k\neq 2$, $D^{\times}$ contains free symmetric
  and unitary pairs with respect to the $k$-involution on $D$ induced
  by the canonical involution on $G$.
\end{theorem}

\begin{proof}
  Take two elements $x$ and $y$ in $G$ satisfying  $[x, y]=\lambda \neq 1$ and
  $[x, \lambda]=[y, \lambda]=1$. Then $\langle x,y\rangle\cong\Gamma$ and 
  $\langle x,y\rangle$ is invariant under the canonical involution of $G$. 
  Now apply Theorem~\ref{T:NilpD} to the subdivision
  $k$-algebra of $D$ generated by $\Gamma$.
\end{proof}

Using somewhat similar methods, we are able to analyze the the following situation.
Let $A_1(\mathbb{Q})=\mathbb {Q} \langle s, t : st-ts=1 \rangle $ denote the 
\emph{first Weyl algebra} 
over the field of rational numbers $\Q$. We shall deal with
the problem of verifying the existence of free symmetric and unitary pairs in the division
ring of quotients of $A_1(\mathbb{Q})$, with respect to an involution which is an extension of one on 
$A_1(\mathbb{Q})$.

One of our main difficulties here is that  we  cannot consider the same kind of 
specializations used in the proof of Theorem~\ref{T:NilpD}.  So, we specialize to 
an algebra of dimension $9$ over its center. This idea yields a proof of the 
following result.

\begin{theorem} \label{T:Weyl}
Let $D$ denote the division ring of quotients of the first Weyl algebra over $\Q$.
Then $D^{\times}$ contains a free pair of the form $\{u, v^{-1}uv\}$, which is
\begin{enumerate}[(i)]
\item symmetric with respect to the involution $\ast$ on $D$ satisfying $s^{\ast}=t$
  and $t^{\ast}=s$, taking $u=\bigl(1+(ts)^{2}\bigr)\bigl(1+(ts +1)^{2}\bigr)^{-1}$ and 
  $v=(1+sts-tst)(1-sts+tst)^{-1}$.\label{Weyl:1}
\item unitary with respect to the involution $\ast$ on $D$ satisfying $s^{\ast}=s$
  and $t^{\ast}=-t$, taking $u=(1+ts+st)(1-ts-st)^{-1}$ and 
  $v=(1-t^{2}s+st^{2})(1+t^{2}s-st^{2})^{-1}$.\label{Weyl:2}
\end{enumerate}
\end{theorem}

The proof of Theorem~\ref{T:Weyl} will be presented in Sections~\ref{S:Inv1}
and \ref{S:Inv2}.

Because $A_1(\mathbb{Q})$ has a $\mathbb{Q}$-automorphism sending $s$ to $-t$ and
$t$ to $s$, Theorem~\ref{T:Weyl} also guarantees the existence in the division ring
of quotients of the first Weyl algebra over $\mathbb{Q}$ of free symmetric pairs 
for the involution 
such that $s^{\ast}=-t, t^{\ast}=-s$, and of free unitary pairs for the involution 
such that $s^{\ast}=-s, t^{\ast}=t$.  


\section{Main tools}\label{S:tools}

In this section we state results that will be used in the proofs of
Theorems~\ref{T:NilpD} and \ref{T:Weyl}.

By a \emph{specialization} from a ring $R$ to a ring $S$ one understands a surjective ring
homomorphism $\alpha\colon R_0\to S$, where $R_0$ is a subring of $R$, such that
$\ker\alpha$ is contained in the Jacobson radical of $R_0$. Such a map will, then,
send non-units of $R_0$ to non-units of $S$. Therefore, any pair of elements of $R_0$ 
mapping to a free pair in $S^{\times}$ will be free in $R^{\times}$.

In the presence of valuations, the next result provides explicit specializations
from division rings of quotients of skew polynomial rings to division rings. We recall that if $k$ is a commutative 
ring and $\sigma$ is an automorphism of $k$, then the skew polynomial ring $k[X;\sigma]$
is the set of polynomials $a_0+Xa_1+\dots+X^na_n$, with $n\in\mathbb{N}$ and coefficients
$a_i\in k$, in the indeterminate $X$ and multiplication satisfying $aX=Xa^{\sigma}$, 
for all $a\in k$. If $k$ is an integral domain, then $k[X;\sigma]$ is an Ore
domain, which, hence, has a division ring of quotients.

\begin{theorem} \label{T:Localization} \cite[Theorem~3.5]{GS11}
Let $R$ be an integral domain with quotient field $F$, and assume that there exists 
a discrete nonarchimedean valuation $\nu$ on $F$ such that $\widetilde{R}=\{x\in F : \nu(x)\geq 0\}$ is the 
valuation ring of $F$ with maximal ideal $A=\{x \in F : \nu(x)>0 \} $. 
Let $a,b \in R$, with $\nu(a)=0$, let $\sigma$ be the $R$-automorphism of the 
polynomial ring $\widetilde{R}[X]$ given by $X^{\sigma}=aX+b$, and let
$Q=\widetilde{R}[X][Y;\sigma]$. Then the complement $M$ of the 
ideal $AQ$ is a right denominator set in $Q$. Further, in the right quotient ring 
$QM^{-1}$ of $Q$ with respect to $M$, the set $B$ of non-units is an ideal, and 
$B\cap Q=AQ$. \qed
\end{theorem}

In the context of Theorem~\ref{T:Localization}, the canonical surjective map 
$QM^{-1}\to QM^{-1}/B$ defines, then, a specialization from the
division ring of quotients $F(X)(Y;\sigma)$ of the skew polynomial ring of 
$F(X)[Y;\sigma]$ to the division 
ring $QM^{-1}/B$.

In the proofs of Theorems~\ref{T:NilpD} and \ref{T:Weyl}, free pairs satisfying the
desired property with respect to the involution under consideration will be shown to 
exist by lifting free pairs, through specializations, from division rings 
where they are known (or proved) to exist. We shall construct specializations to
a quaternion algebra for Theorem~\ref{T:NilpD}, and to a cyclic algebra of degree
three for Theorem~\ref{T:Weyl}.

The existence of free pairs in quaternion and cyclic algebras will follow from
general results (such as those in \cite{GMS99} and \cite{GMS00}) or, via regular representations
of these algebras, with the help of the following criterion of freedom for elements in a 
matrix group, where valuations play again an important role.

\begin{theorem} \label{T:FreeMat} \cite[Theorem~2.10]{GS06}
  Let $\nu$ be a nonarchimedean valuation on a field $K$, let $n \geq 2$ be an integer, 
  and let $A, B \in \GL(n,K)$. Assume the following hold: 
  \begin{enumerate}[(a)]
  \item the matrix $A$ is diagonalizable over $K$ and has a unique eigenvalue with 
	  maximal $\nu$-valuation 
	  and a unique eigenvalue with minimal $\nu$-valuation; and
  \item all the entries of $B$ and of $B^{-1}$ have zero $\nu$-valuation.
  \end{enumerate}
  Then $\{A, B^{-1}AB\}$ is a free pair in $\GL(n,K)$.\qed
\end{theorem}

We shall be considering \emph{right} regular representations of algebras.
Thus, for the sake of clarifying notational issues, we shall
recall some well-known facts. 

Given a field $k$, a $k$-algebra $S$ and a left $S$-module $M$, we shall denote
the set of $S$-module endomorphisms of $M$ by $\End({}_SM)$. If $f,g\in\End({}_SM)$
and $x\in M$, we shall denote the image of $x$ under $f$ by $xf$, and $fg$ will
denote the composed map that satisfies $x(fg)=(xf)g$. Under composition, 
$\End({}_SM)$ is a $k$-algebra.

Suppose, now, that $M$ is free of finite rank with basis $\{x_1,\dots, x_n\}$.
So $M=\bigoplus_{i=1}^n Sx_i$
and there is a $k$-algebra isomorphism
$$
\begin{array}{rcl}
\End({}_SM) & \longrightarrow & M_n(S)\\
f & \longmapsto & [f],
\end{array}
$$
where $[f]$ is given by $[f]=(a_{ij})_{i,j=1,\dots,n}$, with $a_{ij}\in S$ defined
by $x_if=\sum_{j=1}^na_{ij}x_j$, for all $i=1,\dots, n$.

Now, suppose that $R$ is a $k$-algebra given with an algebra homomorphism
$S\to R$. Then $R$ becomes a left (and right) $S$-module and the map
$$
\begin{array}{rcl}
R & \longrightarrow & \End({}_SR)\\
r & \longmapsto & \left\{\begin{array}{rcl}
                        R & \longrightarrow & R\\
												x & \longmapsto & xr
									\end{array}\right.
\end{array}
$$
is an injective algebra homomorphism, called the \emph{right regular representation}
of $R$ over $S$. If, moreover, $R$ is a free left $S$-module with basis $\{x_1,\dots, x_n\}$, then,
composing the two maps above, we get an injective algebra homomorphism from $R$ into
$M_n(S)$, sending $r\in R$ to the matrix $(a_{ij})\in M_n(S)$ such that
$x_ir=\sum_{j=1}^na_{ij}x_j$.

\begin{example}\label{Ex:quat}
Let $F$ be a field, let $a,b\in F^{\times}$, and consider the quaternion algebra
$$
\mathfrak{H}=F\langle \mathbf{i},\mathbf{j} : \mathbf{i}^2=a, \mathbf{j}^2=b, 
\mathbf{i}\mathbf{j}=-\mathbf{j}\mathbf{i}\rangle.
$$
Then $\mathfrak{H}$ is a $4$-dimensional $F$-algebra with basis 
$\{1,\mathbf{i},\mathbf{j},\mathbf{i}\mathbf{j}\}$.
The subalgebra $F[\mathbf{i}]$ of $\mathfrak{H}$ generated by $\mathbf{i}$ is a field and we can 
regard $\mathfrak{H}$ as a free left $L$-module with basis $\{1,\mathbf{j}\}$, where
$L=F[\mathbf{i}]=F(\mathbf{i})$. Then,
the right regular representation of $\mathfrak{H}$ over $L$ affords an
injective $F$-algebra map $\mathfrak{H}\to M_2(L)$
sending $w=\alpha+\beta \mathbf{i}+\gamma \mathbf{j}+\delta \mathbf{i}\mathbf{j} = 
(\alpha+\beta \mathbf{i})1 + (\gamma + \delta \mathbf{i})\mathbf{j}$ to
$$
[w]=\begin{pmatrix} \alpha+\beta \mathbf{i} & \gamma + \delta \mathbf{i} 
\\ b(\gamma - \delta \mathbf{i}) & \alpha-\beta \mathbf{i}
\end{pmatrix}.
$$ 
Letting $K$ denote the subfield $F(\mathbf{j})$ of $\mathfrak{H}$, and viewing $\mathfrak{H}$
as a free left $K$-module with basis $\{1,\mathbf{i}\}$, we get, under the right regular representation,
for $w=\alpha+\beta \mathbf{i}+\gamma \mathbf{j}+\delta \mathbf{i}\mathbf{j}=
(\alpha+\gamma \mathbf{j})1 + (\beta - \delta \mathbf{j})\mathbf{i}
\in\mathfrak{H}$, the
matrix
$$
[w]=\begin{pmatrix} \alpha+\gamma \mathbf{j} & \beta - \delta \mathbf{j} 
\\ a(\beta+\delta)\mathbf{j} & \alpha - \gamma \mathbf{j}
\end{pmatrix}\in M_2(K).
$$
\end{example} 

In Section~\ref{S:One}, we shall use both these right regular representations
of a quaternion in the proof of Theorem~\ref{T:NilpD}.
Likewise, right regular representations of a cyclic algebra of degree $3$ will
be considered in Section~\ref{S:Inv1}.


\section{Involutions on the Heisenberg group}\label{S:InvolG}

For the proof of Theorem~\ref{T:NilpD} we shall need a classification of
involutions on the Heisenberg group.

We say that two involutions $\ast$ and $\dagger$ on a group $G$
are \emph{equivalent} if there exists an automorphism $\varphi$ 
of $G$ such that $\ast\varphi=\varphi\dagger$. This is clearly an
equivalence relation. 

Let $\Gamma$ denote the Heisenberg group, with presentation \eqref{eq:heis}.
We start by mentioning the following properties of $\Gamma$.
In what follows, $\mathcal{Z}(G)$ and $G^{\prime}$ will denote the center and the
commutator subgroup of a group $G$, respectively

\begin{lemma}\label{le:1}
For the Heisenberg group $\Gamma$, with presentation \eqref{eq:heis}, the following
assertions hold.
\begin{enumerate}[(i)]
\item $\mathcal{Z}(\Gamma)=\Gamma^{\prime} = \langle\lambda\rangle \cong \Z$,
  where $\lambda=[x,y]$.\label{p1}
\item $\Gamma/\mathcal{Z}(\Gamma)\cong \Z\times \Z$; an explicit isomorphism
  being given by $x\mathcal{Z}(\Gamma)\mapsto (1,0)$, $y\mathcal{Z}(\Gamma)
	\mapsto (0,1)$.\label{p2}
\item Every element of $\Gamma$ can be written in a unique way in the
  form $\lambda^r y^m x^n$, with $r,m,n\in\Z$. \label{p3}
\item For all $m,n,s\in\Z$, we have 
  $$
  x^n y^m = \lambda^{nm}y^m x^n
  \quad\text{and}\quad
	(y^mx^n)^s = \lambda^{mns(s-1)/2}y^{ms}x^{ns}.
	$$
	\label{p4}
\end{enumerate}
\end{lemma}

Properties \eqref{p1}--\eqref{p3} are well-known and easily proved.
Property \eqref{p4} can be proved directly from the fact that
$\lambda$ is central and that $y^{-1}xy=\lambda x$ and 
$x^{-1}yx=\lambda^{-1}y$.

\medskip

If $\ast$ is an involution on a group $G$ and $N$ is a normal
subgroup of $G$ which is invariant under $\ast$, then $\ast$
induces a unique involution $\overline{\ast}$ on the quotient group $G/N$
satisfying $\ast\pi=\pi\overline{\ast}$, where $\pi\colon G\rightarrow
G/N$ denotes the canonical surjective homomorphism. Since the center $\mathcal{Z}(G)$ 
of $G$ is invariant under $\ast$, it follows that
$\ast$ induces an involution on $G/\mathcal{Z}(G)$. In particular, by Lemma~\ref{le:1},
every involution on $\Gamma$ induces
an involution on $\Z\times \Z$. In other words, every involution
on $\Gamma$ induces an automorphism $\gamma$ of $\Z\times\Z$
satisfying $\gamma^2=1$.

We now look into automorphisms of $\Z\times\Z$ of order $\leq 2$.

There is an isomorphism $\Aut(\Z\times\Z)\cong\GL(2,\Z)$, under which
a matrix $\begin{pmatrix}a&b\\c&d\end{pmatrix}\in \GL(2,\Z)$ 
corresponds to the automorphism
$$
\begin{array}{rcl}
\Z\times\Z & \longrightarrow & \Z\times\Z\\
(m,n) & \longmapsto & (ma+nc, mb+nd).
\end{array}
$$
The following result gives a suitable description of elements of order
$\leq 2$ in $\GL(2,\Z)$.

\begin{lemma}\label{le:2}
  Let $A\in\GL(2,\Z)$ be such that $A^2=I$, the $2\times 2$ identity matrix.
  Then either $A=I$, or $A=-I$, or
  there exists $T\in \GL(2,\Z)$, such that $TAT^{-1}=D$ or $TAT^{-1}=S$, where
  $D=\begin{pmatrix} 1&0\\0&-1\end{pmatrix}$ and $S=
  \begin{pmatrix} 0&1\\1&0\end{pmatrix}$.

Moreover, $I,-I,D$ and $S$ belong to different conjugacy classes in $\GL(2,\Z)$.
\end{lemma}

\begin{proof}
  Since $A^2=I$, its minimal polynomial is either $t-1$, $t+1$ or $t^2-1$. In
  the first two cases, we obtain $A=I$ or $A=-I$. In the remaining case, $\det A=-1$
  and it follows from \cite[Lemma~5.5]{dF89} that there exists $T\in\GL(2,\Z)$
  ($T$ can be taken to be in $\SL(2,\Z)$, actually)
  such that $TAT^{-1}=D$ or $TAT^{-1}=\begin{pmatrix} 1&0\\1&-1\end{pmatrix}$.
  But it is easily seen that $\begin{pmatrix} 1&0\\1&-1\end{pmatrix}$ and $S$
  are conjugate in $\GL(2,\Z)$.

  The last statement follows from easy calculations.
\end{proof}

In order to distinguish among involutions on $\Gamma$ based on the classification
of involutions on $\Gamma/\mathcal{Z}(\Gamma)$ into four classes with respect to equivalence,
as described in Lemma~\ref{le:2}, we shall look for a relation between
automorphisms of $\Gamma$ and automorphisms of $\Gamma/\mathcal{Z}(\Gamma)$.

In general, if $N$ is a characteristic subgroup of a group $G$, then
there is a group homomorphism $\Aut(G)\rightarrow \Aut(G/N)$, defined in
the following way. Let $\pi\colon G\rightarrow G/N$ denote the canonical
epimorphism. Given $\varphi\in\Aut(G)$, the map
$$
G\stackrel{\varphi}{\longrightarrow}G\stackrel{\pi}{\longrightarrow}G/N
$$
factors through $G/N$, because $N^{\varphi}\subseteq N$. The induced
map $\overline{\varphi}\colon G/N\rightarrow G/N$ is an automorphism
of $G/N$ and satisfies $\varphi\pi=\pi\overline{\varphi}$. In other 
words, $\overline{\varphi}$ is given by $(gN)^{\overline{\varphi}}=
g^{\varphi}N$, for all $g\in G$. It is clear that $\varphi\mapsto
\overline{\varphi}$ is a group homomorphism from $\Aut(G)$ into $\Aut (G/N)$.

\begin{lemma}\label{le:3}
  For the Heisenberg group $\Gamma$, the homomorphism $\Aut(\Gamma)\rightarrow 
  \Aut\bigl(\Gamma/\mathcal{Z}(\Gamma)\bigr)$ is surjective.
\end{lemma}

\begin{proof}
  We have seen that to each $\gamma\in\Aut\bigl(\Gamma/\mathcal{Z}(\Gamma)\bigr)$,
  there corresponds a matrix $A=\begin{pmatrix}a&b\\c&d\end{pmatrix} \in\GL(2,\Z)$. 
  Let $\varepsilon=\det A$.
  Then the map $\varphi\colon \Gamma\rightarrow\Gamma$ satisfying
  \begin{align*}
  y^{\varphi} &= \lambda^iy^dx^c,\\
  x^{\varphi} &= \lambda^jy^bx^a,
  \end{align*}
  (and $\lambda^{\varphi}=\lambda^{\varepsilon}$), where 
  $i=-\varepsilon cd(-\varepsilon+a+b)/2$ and $j=-\varepsilon ab(-\varepsilon+c+d)/2$
  is an automorphism of $\Gamma$ (with inverse satisfying $y\mapsto y^{\varepsilon a}
  x^{-\varepsilon c}$, $x\mapsto y^{-\varepsilon b}x^{\varepsilon d}$ and 
	$\lambda\mapsto\lambda^{\varepsilon}$) such that
  $\overline{\varphi}=\gamma$.
\end{proof}

\begin{theorem}\label{T:classinvol}
  Every involution on $\Gamma$ is equivalent to an involution $\ast$ of one of
  the following four types:
	\begin{enumerate}[(I)]
  \item $x^{\ast}=\zeta x, y^{\ast}=\eta y$ (and $\lambda^{\ast}=\lambda^{-1}$), for
    some $\zeta,\eta\in \mathcal{Z}(\Gamma)$;\label{i1}
  \item $x^{\ast}=x^{-1}, y^{\ast}=y^{-1}$ (and $\lambda^{\ast}=\lambda^{-1})$,
    that is, the canonical involution given by the inversion map;\label{i2}
  \item $x^{\ast}=x, y^{\ast}=\zeta y^{-1}$ (and $\lambda^{\ast}=\lambda$), for some
    $\zeta\in \mathcal{Z}(\Gamma)$;\label{i3}
  \item $x^{\ast}=\zeta y, y^{\ast}=\zeta^{-1}x$ (and $\lambda^{\ast}=\lambda$), for
    some $\zeta\in \mathcal{Z}(\Gamma)$.\label{i4}
  \end{enumerate}
\end{theorem}

\begin{proof}
  Direct computation proves that the (involutive extension of the) four maps 
  above are indeed involutions on $\Gamma$.
  
  Let $\star$ be an involution on $\Gamma$ and let $\overline{\star}$ be
  the involution on $\Gamma/\mathcal{Z}(\Gamma)\cong \Z\times \Z$ induced by $\star$.
  By Lemma~\ref{le:2}, it follows that $\overline{\star}$ is equivalent to an 
  involution $\dagger\in\{I,-I,D,S\}$. Thus, there exists an automorphism
  $\gamma\in\Aut\bigl(\Gamma/\mathcal{Z}(\Gamma)\big)$ such that $\overline{\star}
  \gamma=
  \gamma\dagger$. By Lemma~\ref{le:3}, there is $\varphi\in\Aut(\Gamma)$ such
  that $\varphi\pi=\pi\gamma$, where $\pi\colon\Gamma\rightarrow 
  \Gamma/\mathcal{Z}(\Gamma)$
  is the canonical homomorphism. The diagram below depicts all the maps that
  have been considered so far; its four inner quadrilaterals are commutative.
  
  $$
  
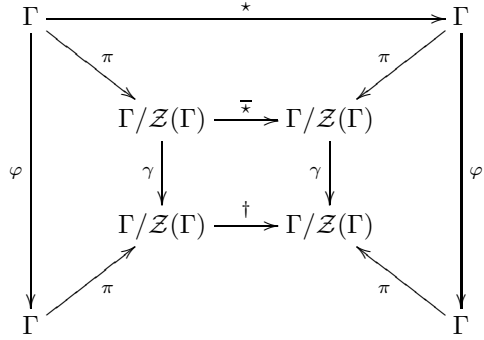
\xymatrix{%
  \Gamma\ar[rrr]^{\star}\ar[dr]^{\pi}\ar[ddd]_{\varphi}&&&\Gamma\ar[dl]_{\pi}
  \ar[ddd]^{\varphi}\\
  & \Gamma/\mathcal{Z}(\Gamma)\ar[r]^{\overline{\star}}\ar[d]_{\gamma}&
  \Gamma/\mathcal{Z}(\Gamma)\ar[d]_{\gamma}&\\
  & \Gamma/\mathcal{Z}(\Gamma)\ar[r]^{\dagger}&\Gamma/\mathcal{Z}(\Gamma)&\\
  \Gamma\ar[ur]_{\pi}&&&\Gamma\ar[ul]^{\pi}
  }
  $$
  
  Consider the involution $\varphi^{-1}\star\varphi$  on $\Gamma$. We shall
  show that the involution it induces on $\Gamma/\mathcal{Z}(\Gamma)$ coincides with
  $\dagger$. Indeed, let $\overline{\varphi^{-1}\star\varphi}$ denote the
  involution on $\Gamma/\mathcal{Z}(\Gamma)$ induced by $\varphi^{-1}\star\varphi$.
  Then,
  \begin{align*}
  \pi\overline{\varphi^{-1}\star\varphi}&=\varphi^{-1}\star\varphi\pi
  =\varphi^{-1}\star\pi\gamma
  =\varphi^{-1}\pi\overline{\star}\gamma
  =\varphi^{-1}\pi\gamma\dagger
  =\varphi^{-1}\varphi\pi\dagger\\
  &=\pi\dagger.
  \end{align*}
  It follows that $\overline{\varphi^{-1}\star\varphi}=\dagger$.
  
  We are left to verify that the only involutions on $\Gamma$ which
  induce an involution $\dagger\in\{I,-I,D,S\}$ on $\Gamma/\mathcal{Z}(\Gamma)$
  are the involutions of the types \eqref{i1}--\eqref{i4}. We shall
  illustrate the argument for the case $\dagger=S$ and leave the remaining
	three cases to the reader.
  
  Suppose, then, that $\ast$ is an involution on $\Gamma$ and that
  $\overline{\ast}=S$. Because $\pi\overline{\ast}=\ast\pi$, it
  follows that there exists $\zeta\in \mathcal{Z}(\Gamma)$ such that $x^{\ast}=\zeta y$.
  Similarly, there exists $\eta\in \mathcal{Z}(\Gamma)$ such that $y^{\ast}=\eta x$.
  Now, applying $\ast$ to the relation $xy=\lambda yx$ and taking into
  account that $\mathcal{Z}(\Gamma)^{\ast}\subseteq \mathcal{Z}(\Gamma)$, we obtain
  $$
  \zeta\eta xy=y^{\ast}x^{\ast}=(xy)^{\ast}=(\lambda yx)^{\ast}=
  \lambda^{\ast}x^{\ast}y^{\ast}=\lambda^{\ast}\zeta\eta yx.
  $$
  Hence, $\lambda^{\ast}=x^{-1}y^{-1}xy=\lambda$. Finally,
  $x=x^{\ast\ast}=(\zeta y)^{\ast} =\zeta y^{\ast}=\zeta\eta x$.
  So, $\eta=\zeta^{-1}$. That is, $\ast$ is of type \eqref{i4}.
\end{proof}


\section{Proof of Theorem \ref{T:NilpD}}\label{S:One}

For the proof of Theorem~\ref{T:NilpD} we shall further explore an idea
developed in \cite{GS11} and construct a specialization, using 
Theorem~\ref{T:Localization}, from the division ring of quotients $D$ of the group
algebra $k\Gamma$ to a quaternion algebra in such a way that a free pair can
be lifted and shown to be either symmetric or unitary with respect to each of
the four types of involutions on $\Gamma$.

More specifically, apply Theorem~\ref{T:Localization} to the polynomial algebra
$R=k[\lambda]$ with quotient field $F=k(\lambda)$ and the valuation $\nu$
on $F$ determined by the prime ideal generated by $1+\lambda$. If $\widetilde{R}$
stands for the valuation ring of $F$, let $Q=\widetilde{R}[x][y;\sigma]$, with
$\sigma$ being the $R$-automorphism of $\widetilde{R}[x]$ satisfying $x^{\sigma}=
\lambda x$. Let $A$, $M$ and $B$ be as in Theorem~\ref{T:Localization}. Letting
$\mathbf{i}$ and $\mathbf{j}$ denote the images of $x$ and $y$, respectively, in the quotient
$QM^{-1}/B$, we obtain that the relation $\mathbf{i}\mathbf{j}=-\mathbf{j}\mathbf{i}$ holds.

Now, let $\mathfrak{H}$ denote the quaternion algebra $k(a,b)\langle \mathbf{i},\mathbf{j} : 
\mathbf{i}^2=a, \mathbf{j}^2=b, \mathbf{i}\mathbf{j}=-\mathbf{j}\mathbf{i}\rangle$ 
over the rational function field $k(a,b)$ in the commuting indeterminates
$a$ and $b$. By what we have seen above,  
we have a surjective homomorphism $\psi\colon QM^{-1}\to\mathfrak{H}$, 
such that $\psi(x)=\mathbf{i}$ and $\psi(y)=\mathbf{j}$, with $\ker\psi=B$. Since $D$ can be
regarded as the division ring of quotients of the skew polynomial ring $k(\lambda,x)[y;\sigma]$,
where $\sigma$ stands for the $k$-automorphism of $k(\lambda,x)$ such that
$\lambda^{\sigma}=\lambda$ and $x^{\sigma}=\lambda x$, it follows that
$\psi$ defines a specialization from $D$ to $\mathfrak{H}$.

Note that we have $\psi(\lambda)=-1$, and so, under $\psi$, the image in 
$\mathfrak{H}$ of an element in the center $\mathcal{Z}(\Gamma)=\langle\lambda\rangle$ 
of $\Gamma$ is either $1$ or $-1$.  

\medskip

\noindent\textit{Existence of free symmetric pairs:} we consider involutions in each
of the four classes of Theorem~\ref{T:classinvol}.

\begin{enumerate}[(I)]
\item The involution on $D$ satisfies $x^{\ast}=\zeta x, y^{\ast}=\eta y$ 
  and  $\lambda^{\ast}=\lambda^{-1}$, for some $\zeta,\eta\in \mathcal{Z}(\Gamma)$.
  Let $\zeta=\lambda^{m}$ and $\eta=\lambda^{n}$, where $m$ and $n$ are integers. 
  There are four distinct cases to consider.
 
  \begin{enumerate}[(i)]
  \item \textit{$m$ and $n$ are both even:} In this case, set $u=1+x+x^{\ast}$ and 
	  $v=1+y+y^{\ast}$. Then  $u=u^{\ast}$, $v=v^{\ast}$, and applying $\psi$ we obtain 
		$\psi(u)=1+2\mathbf{i}$ and $\psi(v)=1+2\mathbf{j}$. By \cite[Theorem~2]{GMS99},
		$\{1+2\mathbf{i}, 1+2\mathbf{j}\}$ is a free pair, and so $\{u, v\}$ is a free 
		symmetric pair.
  \item \textit{$m$ is even and $n$ is odd:} Here, set $u=1+x+x^{\ast}$ and 
	  $r=y-y^{\ast}$. Then $r$ is anti-symmetric and $v=(1-r)(1+r)^{-1}$ is unitary;
		hence, $v^{-1}uv$ is symmetric. Applying $\psi$ we obtain 
		$\psi(v)=(1-2\mathbf{j})(1+2\mathbf{j})^{-1}=(1-4b)^{-1}(1-2\mathbf{j})^{2}$. 
		Again, the pair 
		$\{1+2\mathbf{i}, (1-4b)^{-1}(1-2\mathbf{j})^2\}$ is free, and so is $\{u, v^{-1}uv\}$.
		\label{caseIii}
	\item \textit{$m$ is odd and $n$ is even:}  This case is analogous to \eqref{caseIii}, above.
  \item \textit{$m$ and $n$ are both odd:} In this case, $\ast$ induces an involution of the
	  first kind and of symplectic type on $\mathfrak{H}$ in such a way that $\psi$ is
	  an involution preserving homomorphism. Since there are no free symmetric pairs
	  in $\mathfrak{H}$ in this case (by \cite[Lemma~2.6]{FGM05}), our technique does 
		not apply, and we do not have an answer.
  \end{enumerate}    
  
\item The involution on $D$ satisfies $x^{\ast}=x^{-1}, y^{\ast}=y^{-1}$ and
  $\lambda^{\ast}=\lambda^{-1}$. In this case, set $u=1+x+x^{\ast}$ and 
  $v=1+y+y^{\ast}$. Applying $\psi$, we obtain 
	$\psi(u)=1+(1+a^{-1})\mathbf{i}$ and $\psi(v)=1+(1+b^{-1})\mathbf{j}$. By \cite[Theorem~2]{GMS99},
  these two elements form a free pair in $\mathfrak{H}$. Since $u$ and $v$ are symmetric,
	it follows that $\{u, v\}$ is a free symmetric pair.\label{caseII}

\item The involution on $D$ satisfies $x^{\ast}=x, y^{\ast}=\zeta y^{-1}$ and
  $\lambda^{\ast}=\lambda$, for some $\zeta\in \mathcal{Z}(\Gamma)$. 
	Let $\zeta=\lambda^{m}$. We have two cases to consider.
  
	\begin{enumerate}[(i)]
  \item \textit{$m$ is even:} In this case, set $u=1+x$ and 
    $r=xy^{5}-\zeta^{5}y^{-5}x$.  Then $u^{\ast}=u$ and $r^{\ast}=-r$. So, 
	  letting $v=(1-r)(1+r)^{-1}$, we get that
    $u$ and $v$ are, respectively, a symmetric and a unitary element in $D^{\times}$. 
	  The images of $u$ and $v$ under the specialization $\psi$ are
	  $\psi(u)=1+\mathbf{i}$ and $\psi(v)=(1-\omega \mathbf{i}\mathbf{j})
		(1+\omega \mathbf{i}\mathbf{j})^{-1}$,
	  where $\omega=b^{2}+b^{-3}$. Let us set $F=k(a,b)$, $L=F(\mathbf{i})$ and 
		regard $\mathfrak{H}$ 
	  as a free left $L$-module with basis $\{1, \mathbf{j}\}$. Under the right regular
		representation of $\mathfrak{H}$ over $L$ (see Example~\ref{Ex:quat}), the element $\psi(u)=1+\mathbf{i}$ is 
		represented by the matrix
    $$
    A = \begin{pmatrix} 1+\mathbf{i}&0\\0&-1+\mathbf{i}\end{pmatrix},
    $$
    and the element $\psi(v)=(1+\omega^2 ab)^{-1}(1-\omega^{2}ab-2\omega \mathbf{i}\mathbf{j})$,
    by the matrix
    $$
    B=(1+\omega^{2}ab)^{-1}\begin{pmatrix}
    1-\omega^{2}ab & -2\omega\mathbf{i}  \\
    2\omega b\mathbf{i}         & 1-\omega^{2}ab
    \end{pmatrix}.
    $$
		Let $S$ be the ring of integers of $L$ over $k(b)[a]$. 
    Denoting by $\nu$ the valuation on $L$ determined by the
		prime ideal of $S$ containing $1+\mathbf{i}$, and taking note that 
		$\Norm_{L/F}(1+\mathbf{i})=1-a$,  
		we can check that the 
    matrices $A$ and $B$ satisfy the conditions of 
    Theorem~\ref{T:FreeMat}. Hence, $\{u, v^{-1}uv\}$ is a free 
    symmetric pair.\label{caseIIIi}
  \item \textit{$m$ is odd:} Take $u$ as before and $r=xy -\zeta y^{-1}x$. 
	  Then $r^{\ast}=-r$ and $v=(1-r)(1+r)^{-1}$ is a unitary element in $D^{\times}$. Here 
		$\psi(r)=\mathbf{i}\mathbf{j}-b^{-1}\mathbf{i}\mathbf{j}=(1-b^{-1})\mathbf{i}\mathbf{j}$, and 
		$\psi(v)=(1-(1-b^{-1})\mathbf{i}\mathbf{j})(1+(1-b^{-1})\mathbf{i}\mathbf{j})^{-1}$. As in 
		\eqref{caseIIIi}, we show that 
		$\{u, v^{-1}uv\}$ is a free symmetric pair.
  \end{enumerate}

\item The involution on $D$ satisfies $x^{\ast}=\zeta y, y^{\ast}=\zeta^{-1}x$ and
  $\lambda^{\ast}=\lambda$, for some $\zeta\in \mathcal{Z}(\Gamma)$. Let 
	$\zeta=\lambda^{m}$, with $m$ an integer. There are two cases to consider.

  \begin{enumerate}[(i)]
  \item \textit{$m$ is even:} In this 
    case, take $u=1+x+y^{\ast}$ and $v=1+x^{\ast}+y$. Then $u^{\ast}=v$ and 
    the pair $\{\psi(u), \psi(u^{\ast})\}=\{1+2\mathbf{i},1+2\mathbf{j}\}$ is free by 
    \cite[Theorem~2]{GMS99}. Our free symmetric pair in $D$ is then
    $\{uu^{\ast}, u^{\ast}u\}$. \label{caseIVi}
  \item \textit{$m$ is odd:} Note that, here, if we take the same elements as in \eqref{caseIVi}, 
	  their images collapse to zero. Since $\psi(\zeta)=-1$, let us define 
		$u=1+x+\zeta y^{\ast} =1+2x$. Then $u^{\ast}=1+x^{\ast}+\zeta y=1+2\zeta y$, and we 
		have $\psi(u)=1+2\mathbf{i}, \psi(u^{\ast})=1-2\mathbf{j}$. Therefore, $\{u, u^{\ast}\}$ is a free 
		pair that switches under $\ast$, by \cite[Theorem~2]{GMS99}. So,
    $\{uu^{\ast}, u^{\ast}u\}$ is a free symmetric pair.
  \end{enumerate}
\end{enumerate} 

\medskip

\noindent\textit{Existence of free unitary pairs:} again, we consider the
four classes of involutions on $\Gamma$.

\begin{enumerate}[(I)]
\item The involution on $D$ satisfies $x^{\ast}=\zeta x, y^{\ast}=\eta y$ 
  and $\lambda^{\ast}=\lambda^{-1}$, for some $\zeta,\eta\in \mathcal{Z}(\Gamma)$.
  Let $\zeta=\lambda^{m}$ and $\eta=\lambda^{n}$, where $m$ and $n$ are integers.
	There are two cases to consider.
  
	\begin{enumerate}[(i)]
  \item \textit{$m$ and $n$ are both even or have distinct parity:} In this case,
	  the specialization $\psi$ can be regarded as homomorphism which
		preserves involutions onto the quaternion algebra $\mathfrak{H}$, with
		an involution of the first kind and orthogonal type. Since, by \cite[Lemma~2.6]{FGM05},
		there are no free unitary pairs in this case, our technique does not work.
  \item \textit{$m$ and $n$ are both odd:} 
	  Set $u=(1-x+x^{\ast})(1+x-x^{\ast})^{-1}$ and $v=(1-y+y^{\ast})(1+y-y^{\ast})^{-1}$. 
		Then $\psi(u)=(1-4a)^{-1}(1-2\mathbf{i})^{2}$, 
		$\psi(v)=(1-4b)^{-1}(1-2\mathbf{j})^{2}$, and 
		$\{u, v \}$ is a free unitary pair.
  \end{enumerate}
   
\item The involution on $D$ satisfies $x^{\ast}=x^{-1}, y^{\ast}=y^{-1}$ and
  $\lambda^{\ast}=\lambda^{-1}$. In this case, take $u=x-x^{\ast}$ and $v=y-y^{\ast}$. 
  Then $\psi\bigl((1-u)(1-u^{\ast})^{-1}\bigr)=(1-(1-a^{-1})^2a)^{-1}(1-(1-a^{-1})\mathbf{i})^2$
	and $\psi\bigl((1-v)(1-v^{\ast})^{-1}\bigr)=(1-(1-b^{-1})^2b)^{-1}(1-(1-b^{-1})\mathbf{j})^2$.
	Arguing as in \eqref{caseII} above, we conclude that these two unitary elements 
  generate a free subgroup of $D^{\times}$.
 
\item The involution on $D$ satisfies $x^{\ast}=x, y^{\ast}=\zeta y^{-1}$ and
  $\lambda^{\ast}=\lambda$, for some $\zeta\in \mathcal{Z}(\Gamma)$. 
	Let $\zeta=\lambda^{m}$, $m$ an integer. We need to consider two possibilities.

  \begin{enumerate}[(i)]
  \item \textit{$m$ is even:} In this case, set $v=y-y^{\ast}$;
    then $r=(1-v)(1+v)^{-1}$ is unitary in $D$. 
		Since $\mathbf{j}^{-1}=b^{-1}\mathbf{j}$, we have
    \begin{align*}
    \psi(r)&=\frac{1-\mathbf{j}+b^{-1}\mathbf{j}}{1+\mathbf{j}-b^{-1}\mathbf{j}}=
		\frac{1-(1-b^{-1})\mathbf{j}}{1+(1-b^{-1})\mathbf{j}}\\
		&=\frac{1}{-b+3b^2-b^3}\bigl(b+(1-b)\mathbf{j}\bigr)^2.
    \end{align*}
    Set $u=xy^5-(xy^{5})^{\ast}$; then $s=(1-u)(1+u)^{-1}$ 
    is also unitary in $D$. We have
    \begin{align*}
    \psi(s)&=\psi\left(\frac{1-xy^5+\xi^{5}y^{-5}x}{1+xy^5-\xi^{5}y^{-5}x}\right)
    =\frac{1-b^2\mathbf{i}\mathbf{j}+b^{-3}\mathbf{j}\mathbf{i}}
		{1+b^2\mathbf{i}\mathbf{j}-b^{-3}\mathbf{j}\mathbf{i}}\\
		&=\frac{1+(b^{-3}+b^{2})\mathbf{j}\mathbf{i}}{1-(b^{-3}+b^{2})\mathbf{j}\mathbf{i}} = 
		\frac{\bigl(1+(b^{-3}+b^{2})\mathbf{j}\mathbf{i}\bigr)^2}{1+(b^{-3}+b^{2})^2ab}\\
		&=\frac{1}{b^6(1+(b^{-3}+b^{2})^2ab)}\bigl(b^3+(1+b^5)\mathbf{j}\mathbf{i}\bigr)^2,
    \end{align*}
    Now, instead of proving that $\psi(r)$ and 
		$\psi(s)$ generate a free subgroup, let us prove that $\psi(r)$ and $\psi(s^{-1}rs)$ 
		do. Set $F=k(a,b)$, $L=F(\mathbf{j})$, and let us regard $\mathfrak{H}$ as a left $L$-module with 
		basis $\{1,\mathbf{i}\}$. Then, under the right regular representation of $\mathfrak{H}$ over
		$L$, we obtain
    $$
    b+(1-b)\mathbf{j} \longmapsto A=\mathbf{j}\begin{pmatrix} 1-b+\mathbf{j}  & 0 \\
    0  &  -1+b+\mathbf{j}\end{pmatrix}
    $$
    and
    $$
    b^3+(1+b^5)\mathbf{j}\mathbf{i} \longmapsto B=\mathbf{j}\begin{pmatrix}b^2\mathbf{j}
		& 1+b^5 \\
    -a(1+b^5) & b^2\mathbf{j}\end{pmatrix}.
    $$
		
		We shall show that $A^2$ and $B^2$ satisfy the hypotheses in
		Theorem~\ref{T:FreeMat}.		
    Let $R=k(a)[b]$ and let $S$ be the ring of integers 
    of $L$ with respect to $R$. Let us start by observing that the discriminant of the basis 
    $\{1, \mathbf{j}\}$ is $4b$, and so $b$ is the only prime of $R$ that ramifies in $S$. Also, 
    since $\charac k \neq 2$, the polynomial $f(X)=X^2-b$ is separable over $F$.
    Let us denote $\alpha=-1+b+\mathbf{j}$. Then $\Norm_{L/F}(\alpha)=1-3b+b^2$, $\alpha$ is not a 
    unit of $S$, and $\alpha$ and $-1+b-\mathbf{j}$ are not associates. Therefore, if 
    $\mathfrak{P}$ is the prime ideal of $S$ that contains $\alpha$, and $\nu$ is the 
    valuation associated to $\mathfrak{P}$, we have that $\nu(-1+b-\mathbf{j})=0$. 
		From this it 
    follows that $\nu(\alpha^2)=2$, and that $\nu\bigl((1-b+\mathbf{j})^2\bigr)=0$.
		This proves that $A^2$ is a diagonal matrix with diagonal entries having different
		$\nu$-values.
		
    Let us now check that the valuation $\nu$ on each entry of $B^2$ and of $B^{-2}$ 
    is $0$. This fact for
		$$
    B^2=b\begin{pmatrix}b^5-a\beta{^2} & 2\beta b^{2}\mathbf{j} \\
    -2a\beta b^{2}\mathbf{j} &  b^5-a\beta^{2}\end{pmatrix},
    $$
		where $\beta=1+b^5$, follows from the following:
    \begin{itemize}
    \item $b^5-a\beta^{2}=b^5-a(1+b^5)^2$ and $\nu(b^5-a\beta^{2})=0$.
    \item $\nu(2b^2\beta\mathbf{j})=\nu\bigl(2b^2(1+b^5)\mathbf{j}\bigr)=
		\nu(1+b^5)+\nu(\mathbf{j})+2\nu(b)=0$, 
      since, from the fact that $\gcd\bigl(1+b^5,\Norm_{L/F}(\alpha)\bigr)=1$, it 
			follows that $\nu(1+b^5)=0$. Also, from $\mathbf{j}^2=b$, it follows that 
			$2\nu(\mathbf{j})=\nu(b)=0$.
    \item Similarly, $\nu(-2ab^2\beta\mathbf{j})=0$.
    \end{itemize}
		The argument for the entries of $B^{-2}$ is similar.
    Now, by Theorem~\ref{T:FreeMat}, $\{ A^2,B^{-2}A^2B^2\}$ is a free pair.
    Therefore the unitary elements $r$ and $s^{-1}rs$ of $D^{\times}$ form a free pair.
	\item \textit{$m$ is odd:} Set $v=y-y^{\ast}$ and $r=(1-v)(1+v)^{-1}$. 
	  Then $\psi(r)=(1-(1+b^{-1})^{2}b)^{-1}(1-(1+b^{-1})\mathbf{j})^{2}$. Also set 
		$u=xy-(xy)^{\ast}$. Then $s=(1-u)(1+u)^{-1}$ is unitary and we have 
		$\psi(s)=(1+(b^{-1}-1)^{2}ab)^{-1}(1-(b^{-1})\mathbf{j}\mathbf{i})^{2}$. Repeating the arguments 
		of the former item we can prove that $\psi(r)$ and $\psi(s^{-1}rs)$ generate a 
		free group.
  \end{enumerate}
 
\item The involution on $D$ satisfies $x^{\ast}=\zeta y, y^{\ast}=\zeta^{-1}x$ and
  $\lambda^{\ast}=\lambda$, for some $\zeta\in \mathcal{Z}(\Gamma)$. Let 
	$\zeta=\lambda^{m}$. We consider two cases.
 
  \begin{enumerate}[(i)]
  \item \textit{$m$ is even:} Set $r=xy^{-1}$. Then $r^{\ast}=\zeta^{2}x^{-1}y$.
    Define $u=\bigl(1-(r-r^{*})\bigr)\bigl(1+(r-r^{*})\bigr)^{-1}$. Consider 
	  the specialization $\phi$ from $D$ to $\mathfrak{H}$ such that $\phi(x)=\mathbf{i}$ and 
	  $\phi(y)=\mathbf{k}=\mathbf{i}\mathbf{j}$. Applying $\phi$ to $u$, we obtain  
    $$
    \phi(u)=\frac{1+(1+b^{-1})\mathbf{j}}{1-(1+b^{-1})\mathbf{j}}=
		\frac{-1}{1+b+b^2}(1+b+\mathbf{j})^2.
    $$
    Now, defining  $v=\bigl(1-(x-x^{\ast})\bigr)\bigl(1+(x-x^{\ast})\bigr)^{-1}$, 
	  we obtain
    $$
    \phi(v)=\frac{1-(1+\mathbf{j})\mathbf{i}}{1+(1+\mathbf{j})\mathbf{i}}=
		\frac{1}{1-a+ab}\bigl(1-(1+\mathbf{j})\mathbf{i}\bigr)^2.
    $$
    Regard $\mathfrak{H}$ as a left $k(a,b)(\mathbf{j})$-module with basis $\{1,\mathbf{i}\}$.
	  Under the right regular representation, we get
    $$
    \phi(u)\longmapsto\frac{-1}{1+b+b^2}\begin{pmatrix} 1+b+\mathbf{j}&0\\
		0& 1+b-\mathbf{j}\end{pmatrix}^{2}
    $$
    and
    $$
    \phi(v)\longmapsto \frac{1}{1-a+ab}\begin{pmatrix}1& -1-\mathbf{j} \\a(-1+\mathbf{j}) &  1
		\end{pmatrix}^{2}.
    $$
	  Let
	  $$
    A=\begin{pmatrix} 1+b+\mathbf{j}&0\\0& 1+b-\mathbf{j}\end{pmatrix}
	  $$
	  and
		$$
	  B=\begin{pmatrix}1& -1-\mathbf{j} \\a(-1+\mathbf{j}) &  1\end{pmatrix}.
    $$
		
    We claim that $A^2$ and $B^{-2}A^2B^2$ generate a free group. Indeed,
	  let $F=k(a, b)$, $R=k(a)[b]$,  $L=F(\mathbf{j})$, and let $S$ be the ring of 
    integers of $L$ with respect to $R$. We start by noting that the discriminant 
    of the basis $\{1,\mathbf{j}\}$ is $4b$, and so $b$ is the only prime of 
    $R$ that ramifies in $S$. Also, since $\charac k \neq 2$, the polynomial 
    $f(X)=X^2-b$ is separable over $F$. Let us denote $\mu=1+b+\mathbf{j}$. Then 
    $\Norm_{L/F}(\mu)=1+b+b^2$, $\mu$ is not a unit of $S$, and $\mu$ and 
		$1+b-\mathbf{j}$ are not 
    associates; for if $1+b-\mathbf{j}\in\mathfrak{P}$, where $\mathfrak{P}$ denotes the prime 
	  ideal of $S$ that contains $\mu$, we would have that $2\mathbf{j}$, and hence 
		$\mathbf{j}$ would
	  belong to $\mathfrak{P}$, but, then, $1=\Norm_{L/F}(\mu)-b-b^2 =
	  \Norm_{L/F}(\mu)-(\mathbf{j}^2+\mathbf{j}^4)\in\mathfrak{P}$, a contradiction. 
		Thus, if $\nu$
	  denotes the valuation associated to $\mathfrak{P}$, then $\nu(1+b-\mathbf{j})=0$
	  and $\nu(\mu^2)=2$. Let us check that the valuation $\nu$ on each entry of 
	  $B^2$ and of $B^{-2}$ is $0$.  We have
	  $$
    B^2=\begin{pmatrix}1+a-ab&-2(1+\mathbf{j})\\2a(-1+\mathbf{j})&1+a-ab\end{pmatrix}.
    $$
		Now, $1+a-ab$ is not divisible by $1+b+b^2$ in $k(a)[b]$, so
		$\nu(1+a-ab)=0$. For the other two entries, note that
	  $\Norm_{L/F}(-2(1+\mathbf{j}))=4(1-b)$ and $\Norm_{L/F}(2a(-1+\mathbf{j}))=4a^2(1-b)$,
		which are not divisible by $1+b+b^2$. This implies that these
		two entries also have zero $\nu$-valuation.
    The verification for the entries of $B^{-2}$ is analogous. By 
	  Theorem~\ref{T:FreeMat},  $\{A^2, B^{-2}A^2B^2\}$ is a free pair, and,
    therefore, $\{u, v^{-1}uv\}$ is a free unitary pair.
  \item \emph{$m$ is odd:} This case is very similar to the former. We take the same 
    $u$ and $v$ and perform analogous computations to obtain the result.  
  \end{enumerate}
\end{enumerate}

This ends the proof of Theorem~\ref{T:NilpD}.

\bigskip

\begin{remark}
  It is natural to ask why we take, in Theorem~\ref{T:NilpD}, our division ring to be 
  the division ring of quotients of a torsion free nilpotent group algebra, instead of considering 
  a general division ring generated over its center by a torsion free nilpotent subgroup. 
  One of the reasons is that in this latter case it is not always possible to obtain a 
  specialization from the division ring to a quaternion algebra. For example, take
  $\lambda=\sqrt{2}$, $S=\Q(\lambda, X)[Y,\sigma]$ to be the skew polynomial ring 
  over $\Q(\lambda,X)$ in the indeterminate $Y$, with automorphism $\sigma$ such 
  that $X^{\sigma}=\lambda X$ and $\lambda^{\sigma}=\lambda$, and let 
  $D$ be the division ring of quotients of $S$. Let $R$ be the ring of integers
  of $\Q(\lambda)$. Since $(\sqrt{2}+1)(\sqrt{2}-1)=1$ we have that 
  $\sqrt{2}+1$ is invertible in $R$. Therefore there is no ideal $\mathfrak{P}$ in 
  $R$ containing $\sqrt{2}+1$.
  
  In fact, this general situation was already considered in \cite{GMS00}, in the following 
  context. Let $D$ be a division ring with prime subfield $k$, let $x,y\in D^{\times}$ 
  be such that $\lambda=x^{-1}y^{-1}x y \neq 1$ and with $\lambda$ centralizing 
  both $x$ and $y$. Suppose, moreover, that $D$ is generated by $k$, $x$ and $y$. 
  We call $K=k(\lambda, x)$. Then the action of $y$ on $x$ gives rise to an 
  automorphism $\sigma$ of $K$, such that $\lambda^{\sigma}=\lambda$ 
  and $x^{\sigma}=\lambda x$. Under this setting, the best we can say is stated in 
  \cite[Lemma~24 and Theorem~28]{GMS00}:
  \begin{enumerate} [(1)]
  \item If $\lambda$ is not a root of unity, then $x$ is transcendental over 
    $k(\lambda)$, and $D$ is the division ring of quotients of the skew polynomial ring 
    $K[y; \sigma]$.
  \item If $\lambda$ is algebraic over $k$ but not a root of unity, then $k=\Q$, 
  and there is a specialization from $D$ to a symbol algebra $(a, b: n, \Phi, \theta)$ 
  of positive characteristic, with $a$ and $b$ algebraically independent over the 
  prime field.
  \end{enumerate}
  In this general context, we do not have control over the images of symmetric and 
	unitary pairs of $D$ in the specialization, making it impossible to apply the 
	technique of the proof of Theorem~\ref{T:NilpD}. 
\end{remark}


\section{Involutions on the first Weyl algebra}\label{S:InvWeyl}

In this and the following sections we shall address the problem of finding and 
exhibiting pairs of elements
in the division ring of quotients $D_1(\mathbb{Q})$ of the first Weyl algebra $A_1(\mathbb{Q})$
over the field of rational numbers $\mathbb{Q}$ which generate 
a free subgroup of $D_1(\mathbb{Q})^{\times}$, and are symmetric or unitary with respect to 
an involution on $D_1(\mathbb{Q})$.

Similarly to the group algebra case, we only consider involutions 
on $D_1(\mathbb{Q})$ that are extensions of involutions on $A_1(\mathbb{Q})$. However, here we are not able 
to give a complete description in all situations. This is due, on the one hand, to our 
lack of 
understanding of how to construct free pairs in central simple algebras, and, on the 
other hand,
to the intrinsic difficulty of producing ``reasonable'' symmetric and unitary elements 
with respect to 
the given involutions in $D_1(\mathbb{Q})$. As for the involutions on $D_1(\mathbb{Q})$ that are not extensions of involutions on $A_1(\mathbb{Q})$, they are beyond the grasp of our technique.

It is well-known that the $\Q$-automorphisms of $A_1(\mathbb{Q})$ are products of triangular 
automorphisms and
linear automorphisms with determinant $1$ (see \cite{Dixmier68} and \cite{ML84}). 
A full characterization of involutions on $A_1(\Q)$ might be possible along the
same lines. In fact, this has been suggested in \cite{Moskowiczpp}, where the idea of a proof
is outlined. Below, we concentrate
only on linear involutions, in order to illustrate how the methods used in the previous section
can be applied in this context. 

Given $\alpha,\beta,\gamma,\delta\in\Q$, there exists a unique $\Q$-algebra 
antimorphism 
$\ast\colon\Q\langle{s,t}\rangle\to A_1(\mathbb{Q})$ satisfying
$s^{\ast}=\alpha s+\beta t$ and $t^{\ast}=\gamma s+\delta t$.
Such a map will induce an anti-endomorphism $\ast$ on $A_1(\mathbb{Q})$ if and only if 
\begin{equation}\label{eq:weyl1}
\alpha\delta-\beta\gamma=-1.
\end{equation}
Moreover, $\ast$ will be an involution on $A_1(\mathbb{Q})$ if and only if $s^{\ast\ast}=s$ 
and
$t^{\ast\ast}=t$. The first relation gives
\begin{equation}\label{eq:weyl2}
\begin{cases}
      \alpha^2 + \beta\gamma=1  &   \\
      \beta(\alpha + \delta)=0
\end{cases}
\end{equation}
and the second,
\begin{equation}\label{eq:weyl3}
\begin{cases}
     \gamma(\alpha+\delta)=0  &  \\
     \beta\gamma+\delta^2=1
\end{cases}.
\end{equation}


Since equations \eqref{eq:weyl1}, \eqref{eq:weyl2} and \eqref{eq:weyl3}
are equivalent to
$$
\begin{cases}
  \alpha^2+\beta\gamma=1 & \\
  \alpha=-\delta
\end{cases},
$$
we have the following result.

\begin{proposition}
  Let $\ast$ be a linear operator on the $2$-dimensional $\mathbb{Q}$-vector
	space with basis $\{s,t\}$. Then $\ast$ extends to a $\mathbb{Q}$-involution on
	$A_1(\mathbb{Q})=\mathbb{Q}\langle s,t : st-ts=1\rangle$ if and only if there 
	exist $\alpha,\beta,\gamma\in\mathbb{Q}$ satisfying $\alpha^2+\beta\gamma=1$, 
	$s^{\ast}=\alpha s+\beta t$ and $t^{\ast}=\gamma s-\alpha t$.	\qed
	
\end{proposition}

In particular, both
$$
s^{\ast}=t\quad\text{and}\quad t^{\ast}=s
$$
and
$$
s^{\ast}=s\quad\text{and}\quad t^{\ast}=-t
$$
define $\mathbb{Q}$-involutions on $A_1(\mathbb{Q})$.  In the next two sections, we present
a proof of Theorem~\ref{T:Weyl}, according to which the division ring of quotients of $A_1(\mathbb{Q})$ contains
a free symmetric pair with respect to the first involution above, and
a free unitary pair with respect to the second.


\section{Proof of Theorem~\ref{T:Weyl}(\normalfont\ref{Weyl:1})} \label{S:Inv1} 

Let $\ast$ be the $\Q$-involution of $D_1(\mathbb{Q})$ satisfying $s^{\ast}=t$ and 
$t^{\ast}=s$. Consider in $D_1(\mathbb{Q})^{\times}$ the following elements:
$$
u=\bigl(1+(ts)^{2}\bigr)\bigl(1+(ts+1)^{2}\bigr)^{-1}\quad\text{and}\quad 
v=\bigl(1+(sts-tst)\bigr)\bigl(1-(sts-tst)\bigr)^{-1}.
$$
Our objective is to prove that the pair $\{u, v^{-1}uv\}$ is free symmetric.

Observe that, since $(ts)^{\ast}=ts$ and $(sts)^{\ast}=tst$, the elements $ts$ 
and $sts-tst$ are, respectively, symmetric and anti-symmetric. Therefore $v$ is unitary, 
and $u$ and $v^{-1}uv$ 
are indeed symmetric.

We are left to prove that $\{u, v^{-1}uv\}$ is free.

\begin{lemma}\label{C:One}
Let $D$ be the division ring of quotients of the skew polynomial ring $\Q(X)[Y;\sigma]$, where $\sigma$ stands 
for the $\Q$-automorphism of $\Q(X)$ satisfying $X^{\sigma}=X+1$. Then there is
a $\Q$-algebra isomorphism $\phi\colon D_1(\mathbb{Q}) \to D$ such that $\phi(s)=Y^{-1}X$ 
and $\phi(t)=Y$.
\end{lemma}

\begin{proof}
The $\Q$-algebra map from the free algebra $\Q\langle s,t\rangle$ of rank $2$
into $D$, sending $s$ to $Y^{-1}X$ and
$t$ to $Y$ induces a $\Q$-algebra map $\phi\colon A_1(\mathbb{Q})\to D$ that
satisfies $\phi(s)=Y^{-1}X$ and $\phi(t)=Y$. Since $A_1(\mathbb{Q})$ is simple, $\phi$
is injective and can, therefore, be extended to an isomorphism from $D_1(\mathbb{Q})$
onto $D$.
\end{proof}

We next show that there exists a specialization from $D$ to a cyclic algebra
of degree $3$ over a field of characteristic $3$.

Let $GF(3)$ denote the Galois field with $3$ elements and let $K=GF(3)(a,b)$ 
be the field of rational functions in the commuting indeterminates $a$ 
and $b$ over $GF(3)$. Consider the algebra
$$
\mathfrak{C}=K\langle \mathbf{i},\mathbf{j} : 
\mathbf{j}^3=b, \mathbf{i}^3-\mathbf{i}=a , \mathbf{i}\mathbf{j}=
\mathbf{j}(\mathbf{i}+1) \rangle.
$$
Then $\mathfrak{C}$ is a cyclic (central simple) algebra over $K$ of degree $3$. 
(See, \textit{e.g.}~\cite[Section~7.7]{pC91}.) A basis for $\mathfrak{C}$ over
$K$ is given by $\{1,\mathbf{i},\mathbf{i}^2,\mathbf{j},\mathbf{i}\mathbf{j},
\mathbf{i}^2\mathbf{j},\mathbf{j}^2,\mathbf{i}\mathbf{j}^2,\mathbf{i}^2\mathbf{j}^2\}$. 
Moreover, since $\charac K=3$,
it follows that the relation $\mathbf{j}\mathbf{i}=(\mathbf{i}+2)\mathbf{j}$ 
holds in $\mathfrak{C}$.

Now, apply Theorem \ref{T:Localization} to $R=\mathbb{Z}$ and the
valuation on $F=\mathbb{Q}$ determined by the prime $3$.
Let $\widetilde{R}$ be the valuation ring of $F$ and let
$\sigma$ be the $R$-automorphism of $\widetilde{R}[X]$ such that
$X^{\sigma}=X+1$. Let $A, Q, M$ and $B$ be as in
Theorem~\ref{T:Localization}. Note that the relation 
$Y^{-1}XY \equiv X+1 \pmod{3}$ implies $Y^{-3}XY^{3}\equiv X \pmod{3}$.
Also, $X$, $X+1$ and $X+2$ are the roots of the polynomial $t^{3}-t-X(X+1)(X+2)$. 
In the quotient $QM^{-1}/B\cong \mathfrak{C}$, we see that $a=X(X+1)(X+2)$ and $b=Y^{3}$ 
are commuting  independent 
indeterminates over $GF(3)$ which are contained in the center of $\mathfrak{C}$. We, thus,
have a specialization $\tau$ from $D$ to $\mathfrak{C}$, with domain $QM^{-1}$ 
satisfying $\tau(X)=\mathbf{i}$ and $\tau(Y)=\mathbf{j}$.

At this point, it should be noted that 
$$
\phi(ts)=X\quad\text{and}\quad\phi(sts-tst)=Y^{-1}X^2-XY.
$$
And, therefore,
$$
\tau(\phi(ts))=\mathbf{i}\quad\text{and}\quad\tau(\phi(sts-tst))
=\mathbf{j}^{-1}\mathbf{i}^2-\mathbf{i}\mathbf{j}.
$$

Regard $\mathfrak{C}$ as a free left module over $K(\mathbf{i})$ with basis 
$\mathcal{B}=\{1, \mathbf{j}, \mathbf{j}^{2}\}$, 
and consider the right regular representation afforded by this module.
Since $\tau(\phi(sts-tst))=\mathbf{j}^{-1}\mathbf{i}^{2}-\mathbf{i}\mathbf{j}
=-\mathbf{i}\mathbf{j}+b^{-1}(\mathbf{i}+1)^{2}\mathbf{j}^{2}$, the matrix
$W$ of $\tau(\phi(sts-tst))$ with respect to $\mathcal{B}$ is given by
$$
W=\begin{pmatrix}
        0 & 2\mathbf{i} & b^{-1}(\mathbf{i}+1)^{2} \\
        \mathbf{i}^{2} & 0 & 2\mathbf{i}+1 \\
        2b(\mathbf{i}+1) & (\mathbf{i}+2)^{2} &   0
     \end{pmatrix}.
$$
Clearly, under the same representation, the matrix corresponding to
 $\tau(\phi(u))=(1+\mathbf{i}^2)(1+(\mathbf{i}+1)^2)^{-1}$ is given by
 $$
 U=\begin{pmatrix}
     \frac{1+\mathbf{i}^2}{1+(\mathbf{i}+1)^2} & 0 & 0\\
     0 & \frac{1+(\mathbf{i}+2)^2}{1+\mathbf{i}^2} & 0\\
     0 & 0 & \frac{1+(\mathbf{i}+1)^2}{1+(\mathbf{i}+2)^2}
      \end{pmatrix}.
 $$

If we denote $V=(I+W)(I-W)^{-1}$, the proof of Theorem~\ref{T:Weyl}(\ref{Weyl:1}) will 
be completed once the following lemma is proved.

\begin{lemma}\label{C:Four}
  The matrices $U$ and $V^{-1}UV$ generate a free 
  subgroup in $\GL(3,K(\mathbf{i}))$.
\end{lemma}

\begin{proof}
We will show that $U$ and $V$ are under the conditions of 
Theorem~\ref{T:FreeMat}.

Let $\mathfrak{R}$ be the integral closure in $K(\mathbf{i})$ of the ring of integers 
$GF(3)(b)[a] \subseteq K$, and let $\mathfrak{P}$ be the maximal ideal of $\mathfrak{R}$ 
containing $1+\mathbf{i}^{2}$.

We claim that the minimal polynomial of $1+\mathbf{i}^{2}$ over $K$  is 
$h(t)=t^{3}-2t^{2}+2t-1-a^{2}$. Indeed, let $\alpha_{1}=\mathbf{i}$, 
$\alpha_{2}=\mathbf{i}+1$, 
$\alpha_{3}=\mathbf{i}+2$, and let $s_{1}=\alpha_{1}+\alpha_{2}+\alpha_{3}$,  
$s_{2}=\alpha_{1}\alpha_{2}+\alpha_{1}\alpha_{3}+\alpha_{2}\alpha_{3}$,  
$s_{3}=\alpha_{1}\alpha_{2}\alpha_{3}$  be the elementary symmetric polynomials 
in $\alpha_{1}, \alpha_{2}$ and $\alpha_{3}$. Let us denote by $\widetilde{s_{1}}$, 
$\widetilde{s_{2}}$ and $\widetilde{s_{3}}$ the elementary symmetric polynomials in 
$\alpha_{1}^{2}$, $\alpha_{2}^{2}$ and $\alpha_{3}^{2}$. Then we have 
$\widetilde{s_{1}}=s_{1}^{2}-2s_{2}$, $\widetilde{s_{2}}=s_{2}^{2}-2s_{1}s_{3}$ and 
$\widetilde{s_{3}}=s_{3}^{2}$, and the minimal polynomial of $\mathbf{i}^{2}$ is 
$g(t)=t^{3}-2t^{2}+t-a^{2}$.  Therefore the minimal polynomial of $1+\mathbf{i}^{2}$ is 
$h(t)=g(t-1)=t^{3}-2t^{2}+2t-1-a^{2}$.

The polynomial $h(t)$ computed above is separable over $K$, and therefore the 
prime ideal $\mathfrak{P}$ of $\mathfrak{R}$ that contains $1+\mathbf{i}^{2}$ induces a 
nonarchimedean valuation $\nu$ on $K(\mathbf{i})$.

First, note that $U$ has a unique $\nu$-maximum and a unique 
$\nu$-minimum, because $\nu\left(\frac{1+\mathbf{i}^{2}}{1+(\mathbf{i}+1)^{2}}\right)=1$, 
$\nu\left(\frac{1+(\mathbf{i}+2)^{2}}{1+\mathbf{i}^{2}}\right)=-1$ and
$\nu\left(\frac{1+(\mathbf{i}+1)^{2}}{1+(\mathbf{i}+2)^{2}}\right)=0$.

Now, it is necessary to verify that each of the entries of $V$ and of $V^{-1}$ has 
$\nu$-valuation equal to zero. For that, the entries of $V$ and $V^{-1}$ 
were computed using \textit{Maple 16} (a list of the commands used can be 
found in Section~\ref{sec:maple}, below) and were simplified using the following facts:
\begin{enumerate}[(i)]
\item since $1+\mathbf{i}^{2}  \in \mathfrak{P}$, we have $\mathbf{i}^{2} \equiv -1 
  \pmod{\mathfrak{P}}$;
\item from $\mathbf{i}^{3}-\mathbf{i}-a=0$ we conclude that $\mathbf{i} 
\equiv a \pmod{\mathfrak{P}}$;
\item if $x \in \mathfrak{P} \cap GF(3)(b)[a]$ then $x$ should be divisible 
  by $1+a^{2}$; and
\item since $1+a^{2} \in \mathfrak{P}$, it follows that $a^{2} \equiv -1 
  \pmod{\mathfrak{P}}$.
\end{enumerate}
First, we have
\begin{itemize}
\item $\det(I-W) \equiv (ab)^{-1}(b+2a+2b^{2}) \pmod{\mathfrak{P}}$ and
\item $\det(I+W) \equiv (ab)^{-1}(a+b+b^{2}) \pmod{\mathfrak{P}}$.
\end{itemize}
Now, letting $(u_{ij})= (\det(I-W))V$, we have:
\begin{itemize}
\item $u_{11} \equiv b^{-1}(2b^{2}+b+1) \pmod{\mathfrak{P}}$
\item $u_{12} \equiv b^{-1}+a \pmod{\mathfrak{P}}$
\item $u_{13} \equiv 1+b^{-1}a(b+1) \pmod{\mathfrak{P}}$
\item $u_{21} \equiv -b+1 \pmod{\mathfrak{P}}$
\item $u_{22} \equiv b^{-1}(1+2ab-2b+2ab^{2}) \pmod{\mathfrak{P}}$
\item $u_{23} \equiv 2+b^{-1}a(b+2)  \pmod{\mathfrak{P}}$
\item $u_{31} \equiv b(1+a)+2a \pmod{\mathfrak{P}}$
\item $u_{32} \equiv a+2ab-2b \pmod{\mathfrak{P}}$
\item $u_{33} \equiv b^{-1}(1-b+2ab+2ab^{2}) \pmod{\mathfrak{P}}$,
\end{itemize} 
and letting $(s_{ij})= (\det(I+W))V^{-1}$, we have:
\begin{itemize}
\item $s_{11} \equiv b^{-1}(2+b+ab^{2}) \pmod{\mathfrak{P}}$
\item $s_{12} \equiv b^{-1}+2a \pmod{\mathfrak{P}}$
\item $s_{13} \equiv 1+ab^{-1}(b+2) \pmod{\mathfrak{P}}$
\item $s_{21} \equiv 2-b \pmod{\mathfrak{P}} $
\item $s_{22} \equiv b^{-1}(2-2b+2ab+ab^{2}) \pmod{\mathfrak{P}}$
\item $s_{23} \equiv 1+ab^{-1}(2b+2)\pmod{\mathfrak{P}}$
\item $s_{31} \equiv 2a+2b(1+a) \pmod{\mathfrak{P}} $
\item $s_{32} \equiv a(2+2b(1+a)) \pmod{\mathfrak{P}}$
\item $s_{33} \equiv b^{-1}(2-b+2ab+ab^{2}) \pmod{\mathfrak{P}}$.
\end{itemize}

Therefore, the hypotheses of Theorem~\ref{T:FreeMat} are
satisfied and the lemma follows.
\end{proof}


\section{Proof of Theorem~\ref{T:Weyl}(\normalfont\ref{Weyl:2})} \label{S:Inv2}

Let $\ast$ be the $\Q$-involution of $D_1(\mathbb{Q})$ such that $s^{\ast}=s$ and 
$t^{\ast}=-t$, and let 
$$
u=(1+ts+st)(1-ts-st)^{-1}\quad\text{and}\quad
v=(1-t^{2}s+st^{2})(1+t^{2}s-st^{2})^{-1}.
$$
Then $v^{\ast}=v^{-1}$ and $u^{\ast}=u^{-1}$.

We shall prove that  $\{u, v^{-1}uv\}$ is a  free unitary pair.

We have $\tau(\phi(u))=2(\mathbf{i}+1)\mathbf{i}^{-1}=2a^{-1}(\mathbf{i}+1)^2
(\mathbf{i}+2)$, and $\tau(\phi(v))=(1+2\mathbf{j})(1+\mathbf{j})^{-1}=
(1+b)^{-1}(1+2b+\mathbf{j}+2\mathbf{j}^2)$, and the corresponding matrices with respect 
to the basis $\{1, \mathbf{j}, \mathbf{j}^{2}\}$ are, respectively,

$$
U=\frac{2}{a}\begin{pmatrix}
  (\mathbf{i}+1)^2(\mathbf{i}+2)&0&0\\
  0&\mathbf{i}^2(\mathbf{i}+1)&0\\
  0&0& (\mathbf{i}+2)^2\mathbf{i}
  \end{pmatrix}
$$

and

$$
V= \frac{1}{1+b}\begin{pmatrix}
1+2b   &   1         &     2   \\
2b       & 1+2b   &      1   \\
b       &    2b      &  1+2b
\end{pmatrix}.
$$

It follows that
$$
V^{-1}=\frac{1}{2+b}\begin{pmatrix}
2+2b   &    1       &      1    \\
b        & 2+2b     &     1    \\
b         &   b       & 2+2b
\end{pmatrix}.
$$

Let $\nu$ be the valuation determined by the ideal $\mathfrak{P}$ 
of $\mathfrak{R}$ that contains $\mathbf{i}$. Since $f(t)=t^{3}-t-a$ is 
separable over $K$, it follows that $\mathfrak{P}$ determines  
a nonarchimedean valuation $\nu$ in $K(\mathbf{i})$. Since $U$ is 
$\nu$-max/min and $V^{\pm 1}$ have zero $\nu$-valuation, we are under the 
conditions of Theorem~\ref{T:FreeMat}. The conclusion follows.


\section{Maple Computations} \label{sec:maple}

In order to give a brief indication of how \textit{Maple 16} was used to
perform the computations mentioned in Section~\ref{S:Inv1}, we list
below the package and commands used:

\begin{verbatim}
with(LinearAlgebra)
alias(i=RootOf(x^3-x-a,x) mod 3)
W:=Matrix([[0,2*i,(1/b)(i+1}^2],[i^2,0,2*i+1],[2*b*(i+1),(i+1)^2,0]])
Id:=Matrix([[1,0,0],[0,1,0],[0,0,1]])
A:=Id-W;
B:=Id+W;
eqns:={i=a,a^2=-1};
V:=B.Adjoint(A);
simplify(V,eqns) mod 3;
VV:=A.Adjoint(B);
simplify(VV,eqns) mod 3;
dA:=Determinant(A,method=algnum) mod 3;
simplify(dA,eqns) mod 3;
dB:=Determinant(B,method=algnum) mod 3;
simplify(dB,eqns) mod 3;
\end{verbatim}


\section{Final remarks} \label{S:Rem}

We were not able to exhibit either free symmetric or unitary pairs for 
some of the involutions considered in the previous sections. So, we leave the 
following 

\begin{problem}
Construct for each of the involutions not covered by our proofs free symmetric 
and unitary pairs in $D^{\times}$, where $D$ stands for the division ring of quotients
of either $k\Gamma$ or $A_1(\mathbb{Q})$.
\end{problem}

\section*{Acknowledgment}
We are indebted to Prof.~Donald S.~Passman for pointing gaps in our reasoning 
in earlier versions of this paper.

\end{document}